\documentclass[a4paper,12pt]{amsart}
\usepackage{amsmath,amssymb}
\usepackage[shortlabels]{enumitem}

\usepackage{amsfonts}
\usepackage{mathtools}
\usepackage{dsfont}
\usepackage{esint}
\usepackage{amssymb,latexsym}
\usepackage{xcolor}

\newtheorem{theorem}{Theorem}[section]

\newtheorem{proposition}[theorem]{Proposition}
\newtheorem{lemma}[theorem]{Lemma}

\newtheorem{corollary}[theorem]{Corollary}

\newenvironment{proof*}{\vskip 2mm\noindent {}}{\hfill $\Box$ \vskip 2mm}

\numberwithin{equation}{section}

\def\R{\mathbb R}
\def\C{\mathbb C}
\def\D{\mathbb D}
\def\Z{\mathbb Z}

\renewcommand{\Im}{\operatorname{\rm{Im}}}
\renewcommand{\Re}{\operatorname{\rm{Re}}}

\def\n{\noindent}
\def\fr{\frac}

\def\Om{\Omega}
\def\pa{\partial}

\def\de{\delta}
\def\al{\alpha}

\def\ve{\varepsilon}
\def\eps{\varepsilon}
\def\va{\varphi}
\def\la{\lambda}

\DeclareMathOperator{\dist}{dist}

\title[Removable sets for pseudoconvexity]{Removable sets for pseudoconvexity for weakly smooth boundaries}

\author{Nguyen Quang Dieu}
\address{Nguyen Quang Dieu\\ Dai Hoc Su Pham Ha Noi, Ha Noi, Vietnam,
 \&  Thang Long Institute of Mathematics and Applied Sciences, Thang Long University, 
 Nghiem Xuan Yem, Hoang Mai, Ha Noi, Vietnam}

\email{ngquang.dieu@hnue.edu.vn}

\author{Pascal J. Thomas}
\address{P.J. Thomas\\
Univ Toulouse, CNRS,
Institut de Math\'ematiques de Toulouse, UMR5219 \\
 Toulouse, France}

\email{pascal.thomas@math.univ-toulouse.fr}

\thanks{Most of this work was completed during the visit of the second author at the Vietnam Institute for Advanced Studies in Mathematics, March 2025, in the framework of the project ``Complex Hessian equations on domains in $\C^n$''. }

\begin{document}


\begin{abstract} 
We show that for bounded domains in $\C^n$ with $\mathcal C^{1,1}$ smooth boundary, if
there is a closed set $F$ of $2n-1$-Lebesgue measure $0$ such that $\partial \Om \setminus F$ is $\mathcal C^{2}$-smooth
and locally pseudoconvex at every point, then $\Omega$ is globally pseudoconvex.  Unlike in the globally
$\mathcal C^{2}$-smooth case, the condition ``$F$ of (relative) empty interior'' is not enough to 
obtain such a result. We also give some results under peak-set type hypotheses, which in particular
provide a new proof of an old result of Grauert and Remmert about removable sets for pseudoconvexity under
minimal hypotheses of boundary regularity.
\end{abstract}

\dedicatory{This work is dedicated to Professor Nguy$\, \tilde{{\hat {e}}}$n V\u{a}n Khu\^e, the father of the first-named author, on the occasion of his 85th birthday.}

\maketitle

\section{Introduction and results}

The natural domains of existence of holomorphic functions in several variables, or \emph{domains of holomorphy,}
are characterized by \emph{pseudoconvexity}, which turns out to be a local condition on their boundaries.  
In the case where said boundary is $\mathcal C^{2}$-smooth, this can be understood as a sort of infinitesimal
plurisubharmonicity along the complex tangential directions, and written explicitly in terms of the Levi form of 
the defining function of the domain.

It is natural to ask whether the local condition can be relaxed in the spirit of the Riemann removable singularity
Theorem, so that local pseudoconvexity at ``most'' of the points of the boundary would be enough to conclude 
global pseudoconvexity. As early as 1956, Hans Grauert and Reinhold Remmert wondered ``ob es vielleicht gen\"ugt,
die Pseudokonvexit\"at von $G$ nur in einer 'gen\"ugend dicht liegenden Randpunktmenge von $G$' vorauszusetzen''
\cite[p. 152]{GR}\footnote{``Whether maybe it is enough to assume the pseudoconvexity of $G$ only in a 'sufficiently
dense subset of boundary points of $G$' ''.}. 

Let us say that a set $F \subset \partial \Omega$ is \emph{removable} (for pseudoconvexity,
under some regularity assumption on $\Omega$) if the fact that $\partial \Om \setminus F$ be locally
pseudoconvex at each of its points implies that $\Om$ is pseudoconvex.
Obviously some hypothesis must be made to avoid situations such as $\Omega:= B(0,1) \setminus \{0\} \subset \C^n$,
$n\ge 2$: the boundary is pseudoconvex at all points except $0$, yet the domain clearly is not pseudoconvex. 

Grauert and Remmert then define two notions of smallness,  \emph{slim} (``d\"unn'') sets 
and \emph{erasable} (``hebbar'') boundary points (the English terminology we have chosen is designed
to avoid confusion with other adjectives used in mathematics rather than be a literal translation from the German). 
Namely, a set $A\subset \partial \Omega$ will be called slim if for each $a\in A$ there exists an open neighborhood
$U$ of $a$ and $f\in \mathcal O(U\cap\Omega)$, $f \not \equiv 0$ such that $A\cap U \subset \overline{f^{-1}\{0\}}$ (loosely said, $A$ is locally
contained in the cluster set of the zeros of an analytic function).  A boundary point $p\in \partial \Omega$
will be called erasable if there exists a connected neighborhood $U$ of $p$ and $g \in \mathcal O(U)$, $g \not \equiv 0$,
such that $U\cap \partial \Omega \subset g^{-1}\{0\}$. In the above example of the punctured ball, 
the origin is clearly an erasable boundary point.

Grauert and Remmert prove that if $F \subset \partial \Omega$ is slim, and no point  $a\in F$
is erasable (with respect to $\partial \Omega$) then $F$ is removable \cite[Resultat 1., p. 153]{GR}. 

Here the regularity
assumption on $\partial \Omega$ is modest (roughly speaking, there should be no small components of the boundary,
at least near the exceptional set $F$),
but the requirement on $F$ is stringent: it should be rather small, lying inside the
cluster set of an analytic set. Similarly, we prove that when
we are provided with holomorphic or plurisubharmonic functions peaking on $F$, removability follows,
see Section \ref{struc}.  

At the other end of the spectrum, if we assume that $\partial \Om$ is $\mathcal C^{2}$-smooth, then any 
closed set of empty interior relative to $\partial \Om$ is removable (and those can be rather big), see 
Proposition \ref{c2}. As soon as 
we are looking at sets with some interior, no amount of regularity will cause those to be removable (one can always
have a small concavity inside the set).

It is thus interesting to see which classes of sets can be removable when $\partial \Om$ has some intermediate
regularity. 
One instance of such a result is to be found in \cite{ZZ},
where $\partial \Om$ is assumed to be given by a defining function
which is almost everywhere $\mathcal C^2$ with bounded second
derivatives (which is a wider class than $\mathcal C^{1,1}$),
but also with some shape restrictions. Here the exceptional set
(the points where no pseudoconvexity of the appropriate sort is required) has to be
assumed to be of zero $2n-2$-dimensional Hausdorff measure.

This question of the size of removable sets for pseudoconvexity
in terms of boundary regularity of the domain
is connected to a more recent question, the study of visibility 
of Kobayashi (almost) geodesics in domains of $\C^n$: Banik \cite{Ban} showed that there can be non-pseudoconvex
domains enjoying the visibility property which are typically domains with slim components in 
their boundary, while in \cite{NOT} it was shown that when the boundary of a domain
$\Om$ is $\mathcal C^{2}$-smooth, visibility does imply pseudoconvexity.  One would want to know which regularity
is optimal for one or the other situation, and a way to look for (necessarily less regular)
examples of non-pseudoconvex visibility domains
is to have their boundary strictly pseudoconvex outside of a (necessarily) non-removable set.  Any family of geodesics potentially
violating visibility would have to cluster on that exceptional set, and if it can be made small enough, one can hope
to obtain a contradiction.

Let $\lambda_m$
denote the Lebesgue measure on $\R^m$, and $B(a,r)$ the open ball of center $a$ and radius $r$.

\begin{theorem}
\label{mainsobo}
Let $\Omega= \{z\in\C^n: \rho(z)<0\}$ be a bounded domain, where $\rho$ is a continuous real-valued
function on $\C^n$ such that there is a neighborhood $\mathcal V$ of $\{ \rho(z)=0\}= \partial \Omega$ with
$\rho \in \mathcal C^{1,1}(\mathcal V)$ and
$|\nabla \rho (z)|\ge c_0 >0$, $z\in \mathcal V$.

Suppose that there is a closed set $F\subset \partial \Omega$
such that 
\begin{itemize}
\item
$\pa \Om \setminus F$ is $\mathcal C^2$-smooth;
\item
$\partial \Omega \setminus F$ is locally pseudoconvex, i.e. for any $\zeta \in \partial \Omega \setminus F$,
there exists $r>0$ such that $B(\zeta,r) \cap \Omega$ is pseudoconvex;
\item
$\lambda_{2n-1}(F)=0$.
\end{itemize}
Then $\Omega$ is pseudoconvex.
\end{theorem}
This is proved in Section \ref{n2} in the case of complex dimension $2$, and in Section \ref{n3} in general.


We may ask whether the situation in the $\mathcal C^{1,1}$ case is similar to the $\mathcal C^2$ case.	
It is in fact quite different, and sets with empty (relative) interior may fail to be removable. 

\begin{proposition}
\label{c11nonpc}
There exists a bounded, non-pseudoconvex domain $\Omega \subset \C^2$ with a $\mathcal C^{1,1}$-smooth boundary,
and a closed set $E\subset \partial \Omega$ such that $\partial \Omega \setminus E$ is  dense in $\partial \Omega$, 
$\mathcal C^{2}$-smooth and strictly pseudoconvex in a neighborhood of each of its points.
\end{proposition}

Our result should be considered a first step toward a more systematic study of removable sets.
Many results are known about the analogous question of removable singularities for subharmonic 
and plurisubharmonic functions; their general drift is that the more regular the function, the
larger the removable sets are allowed to be.  For instance, in $\C^n$, sets of Hausdorff measure $0$
in (real) dimension  $2n  - 1 + \alpha$ will be removable for plurisubharmonic functions of class
$\mathcal C^{1,\alpha}$, sets of Hausdorff measure $0$
in (real) dimension $2n - 2 + \alpha$ removable for plurisubharmonic functions 
in the H\"older class $\Lambda^\alpha$ \cite{DD}. It would be interesting to have 
such results for removable sets for pseudoconvexity; for instance, if we assume that $\partial \Om$ is globally $\mathcal C^{1,\alpha}$-smooth,
with $\alpha \in (0,1)$, are closed sets of Hausdorff measure $0$
in (real) dimension  $2n -2+\alpha$  removable? Examples given
in subsection \ref{zyg}, in the case $n=2$, show that there exist sets of Hausdorff dimension equal to $2+\alpha$ and
positive $(2+\alpha)$-Hausdorff measure   which are not removable for  $\mathcal C^{1,\alpha}$-smooth
boundaries.

\section{Use of structure hypotheses}
\label{struc}

\begin{proposition}
\label{c2}
Assume that $\pa \Om$ has $\mathcal C^2$-smooth boundary near a closed subset $F \subset \pa \Om$ and that $\pa \Om \setminus F$	
is locally pseudoconvex. Suppose that $F$ has empty interior relative to $\partial \Om.$ Then $\Om$ is pseudoconvex.
\end{proposition}

\begin{proof}
It suffices to check the Levi condition at a given point $a \in F.$ 
Fix a $\mathcal C^2$ defining function $\rho$ of $\Om$ near $a,$ let $L\rho(a)$ be the 
hermitian form defined by $\left( \frac{\partial^2 \rho}{\partial z_j \partial \bar z_k} (a)\right)_{1\le j,k \le n}$.
It depends continuously on $a$.  For any vector $v \in T^{\C}_a \partial \Om$ (the complex tangent space
to the boundary at $a$), there are sequences $(a_p) \subset \partial \Om$ and $v_p \in T^{\C}_{a_p} \partial \Om$
such that $a_p\to a$, $v_p \to v$.  By continuity, $0 \le L\rho (a) (v,v) = \lim_p L\rho (a_p) (v_p,v_p)$.
\end{proof}

We proceed with a situation where the boundary contains enough peak plurisubharmonic functions.

\begin{proposition} 
Suppose that at every point $a \in \partial \Om \setminus F,$ there exist a neighborhood $U_a$ of $a$ 
and a negative plurisubharmonic function $\tilde u_a$ on $U_a \cap \Om$ such that 
$$
\lim_{z \to a} \tilde u_a (z)=0, \limsup_{z \to \xi} \tilde u_a (z)<0, \ \forall \xi \in \partial \Om \setminus \{a\}.
$$
Assume that there exists a neighborhood $U$ of $F$ and a plurisubharmonic function $v$ on $U$ satisfying the following properties:
 
$$v<0 \ \text{on}\ U \cap \Om, \lim\limits_{z \to F} v(z)=0,\limsup\limits_{z \to \xi} v (z)<0, \ \forall \xi \in \partial \Om \setminus F.$$
	Then $\Om$ is hyperconvex.
\end{proposition}
\n
{\bf Remarks.} 1. The existence of $\tilde u_a$ is obviously confirmed if $\partial \Om$ is strictly pseudoconvex at $a$.

\n
2. This should be compared to  \cite[Propositions 1.2 and 1.3]{NPTZ}.  The exceptional set in the boundary
of the domain constructed in \cite[Lemma 2.3]{NPTZ} cannot be a peak set for such a plurisubharmonic function $v$.

\begin{proof}
{\it Step 1.} Fixed a closed ball $B$ inside $\Om$. 
We construct a global peak plurisubharmonic function at each point $a \in \partial \Om \setminus F$ which 
is constant equal to $-1$ on $B.$
	Take $\tilde u_a$ as in  the assumption, and  $r_a>0, \de_a>0$ so small such that $B \cap B(a,r_a)=\emptyset$ and
	$$
	-\de_a> \sup_{\Om \cap \partial B(a,r_a)} \tilde u_a.
	$$
	Now we set
	$$u_a:= \begin{cases}
	\max\{ \frac1{\de_a} \tilde u_a, -1\} & B(a,r_a) \cap \Om\\
	-1 &\Om \setminus B(a,r_a).
	\end{cases}$$
	Then $u_a$ is the desired peak function.
	
	\n 
	{\it Step 2.} We build a global peak psh function at $\partial \Om \setminus F$.
	Set 
	$$u(z):= (\sup \{u_a (z): a \in (\partial \Om) \setminus F\})^*, z \in \Om,$$
	where $v^*$ denotes the upper semicontinuous regularization of a
	function $v$. 
	Then $u$ is plurisubharmonic on $\Om$, $-1 \le u<0$ on $\Om, u=-1$ on $B$ and $u(z) \to 0$ as $z \to a$, for any $a  \in (\partial \Om) \setminus F$.
	
	\n 
	{\it Step 3.} We patch $v$ and $u$ to yield a global negative exhaustion function for $\Omega.$ To this end, by shrinking
	$U$ we may take $\la>0$ so small such that  $\la u \ge -\la \ge  v$ on $\partial U \cap \Om.$
	Let 
$$
\va:= 
\begin{cases} 
	\max \{\la u, v\}& \ \text{on}\ U \cap \Om,\\
\la u& \ \text{on}\ \Om \setminus U.
\end{cases}
$$
	Then $\va$ is the desired negative psh exhaustion function for $\Om.$
\end{proof}

\begin{proposition}	
\label{nhdU}
	Let $\Om$ be a bounded domain in $\C^n$ and $F \subset \pa \Om$ be a non-empty closed set satisfying the following properties:
\begin{enumerate}
\item
$\pa \Om \setminus F$ is locally pseudoconvex;
\item
There exists a pseudoconvex open neighborhood $U$ of $F$ in $\C^n$ and a continuous plurisubharmonic function $v$ on $U \cap \Om$ such that 
\[
F=\{z \in \pa \Om: \lim_{\xi \to z} v(\xi)=+\infty\}.
\]
\end{enumerate}
	Then $\Om$ is pseudoconvex.
\end{proposition}	

\begin{proof}
	Let $\Om':=U \cap \Om.$ It suffices to show $\Om'$ is pseudoconvex. 
	Set $$\Om_j:= \{z \in \Om': v(z)<j\}, v_j (z):=\fr1{j-v (z)}, \ z \in \Om_j.$$
	We claim that $\partial \Om_j$ is locally pseudoconvex at every boudary point $p \in \partial \Om_j$. Indeed, this is true when $v(p)=j$,
	since we have $v_j$ is plurisubharmonic on $\Om_j$ and tends to $+\infty$ at every boundary point near $p.$
	If $p \in \partial \Om,$ then by $(2)$ we have $p \notin F.$	Let $V$ be a small 
	neighborhood of $p$ such that $V \cap \Om$ is pseudoconvex. Then, we can find a plurisubharmonic exhaustion function $\va$ for $V \cap \Om \cap U.$ So $\max \{\va, v_j\}$ is a plurisubharmonic exhaustion function for $V \cap \Om_j.$ If $p \in \partial U,$ an easier version of the same reasoning holds.
	
	So we have proved that $\Om_j$ is locally pseudoconvex. Thus $\Om_j$ is pseudoconvex. Since 
	$\Om_j \uparrow \Om',$ we infer that $\Om'$ is pseudoconvex as well.
\end{proof}
As a consequence of this result we recover a weaker version of the theorem of Grauert-Remmert \cite{GR} mentioned in the introduction.
\begin{corollary} 
\label{alaGR}	
Let $\Om$ be a bounded domain in $\C^n$ and $F \subset \pa \Om$ be a non-empty closed set satisfying the following properties:
\begin{enumerate}
\item
$\pa \Om \setminus F$ is locally pseudoconvex;
\item
There exists a holomorphic function $h$ on $\Om, h \not \equiv 0$ such that 
\newline
$\lim_{z \to F} h(z)=0.$
\item
For any $\zeta \in F$, any neighborhood $U$ of $\zeta$, and any function $g$ holomorphic in $U\cap \Omega$, $g \not \equiv 0$,
then $U \cap \partial \Om \not \subset \overline{g^{-1}(0)}$.
\end{enumerate}
Then $\Om$ is pseudoconvex.
\end{corollary}

\begin{proof}
	We split the proof into two steps.
	
	\n 
	{\it Step 1.} Set $Z:=\{z \in \Om: h(z)=0\}.$ We claim that $\Om':=\Om \setminus Z$ is pseudoconvex. 
	For this, it suffices to apply the Proposition \ref{nhdU} to $\Om'$, with $U$ a big ball containing $\bar \Om$ and $v:=1/|h|.$
	
	\n 
{\it Step 2.} Consider $(\hat \Om,  \pi)$ the envelope of holomorphy of $\Om$, which is a priori a Riemann domain
over $\C^n$, with $\pi$ the projection to $\C^n$, a local homeomorphism;
denote by $\alpha$
the canonical immersion $\Om \longrightarrow \hat \Om$.

Let $\hat h$ be the extension of $h$ to a holomorphic function
on $\hat \Om$, i.e. $\hat h \circ \alpha =h$. Then, if we let
$\hat Z:= \hat h^{-1}(0)$, $Z=\alpha^{-1}(\hat Z)$.

Then \cite[Satz 7]{GR} states that 
$(\hat \Om \setminus \hat Z, \pi|_{\hat \Om \setminus \hat Z})$
is the hull of holomorphy of $\Om \setminus Z$.


But $\Om \setminus Z$ is pseudoconvex, so its envelope of holomorphy is equal to itself and $\alpha|_{\Om \setminus Z}$
can be taken to be the identity. Using that identification, $\hat \Om \setminus \hat Z = \Om \setminus Z$.
In particular, 
\[
\overline{\hat \Om} \subset \overline{\hat \Om \setminus \hat Z} = \overline{\Om \setminus Z}= \overline \Omega.
\]
As a consequence, $\hat \Omega \setminus \Om \subset \overline \Omega\setminus \Om=\partial \Om$.  Let $p\in \hat \Omega \setminus \Om$,
then by the Grauert and Remmert theorem stated above, 
$p \in \hat \Omega= \hat Z \cup (\Om \setminus Z)\hat{} $, but $(\Om \setminus Z)\hat{} \subset \Om$
so $p\in \hat Z \subset \overline Z$.  Let $U$ be a neighborhood of $p$ in $\hat \Omega$, if $q \in U\cap \partial \Om$,
then $q\notin \Om$, so $q \in \hat \Om \setminus \Om$, so by the above $q \in \overline Z$, a contradiction with
the third hypothesis.
\end{proof}

\section{The case of dimension $2$}
\label{n2}

\subsection{Local expression of the Levi condition.	}

With the notations of Theorem \ref{mainsobo}, in an open subset where $\rho$ is of class
$\mathcal C^2$,  the boundary of $\Omega$ is pseudoconvex if
and only if 
\begin{equation} 
\label{diffpc}
 \fr{\pa \rho^2}{\pa z_1 \pa \bar z_1} \left \vert \fr{\pa \rho}{\pa z_2} \right \vert^2+
\fr {\pa \rho^2}{\pa z_2 \pa \bar z_2}  \left \vert \fr{\pa \rho}{\pa z_1} \right \vert^2
- 2\Re \Big (\fr{\pa^2 \rho}{\pa z_1 \pa \bar z_2} \fr{\pa \rho}{\pa \bar z_1}\fr{\pa \rho}{\pa z_2} \Big ) \ge 0.
\end{equation}

Since by hypothesis $\rho$ is of class $\mathcal C^1$ 
and its gradient never vanishes near $a \in \partial \Omega$, we may choose local coordinates so that 
$a=(0,0)$ and $\nabla \rho (a) = (1,0)$. Applying the Implicit Function Theorem, 
we obtain $U$, $U'$  open neighborhoods of $0$ in $\C^2$ and $ \R \times \C$ and a new defining function
\begin{equation} 
\label{localdef}
\rho_1(z_1,z_2):= x_1-\va (y_1, z_2), \va \in C^1 ( U'), \va(0,0)=0, \nabla\va(0,0)=0,
\end{equation}
such that $$\{z\in U: \rho(z)<0\} = \{ (x_1+iy_1,z_2): (y_1,z_2)\in U', x_1<\va (y_1, z_2)\}.$$
Let us assume that $U$ is always chosen convex, so that the pseudoconvexity of $\Omega \cap U$ only depends
on the boundary of $\Omega$. 

\vskip.5cm

Let $$F'= \{(y_1,z_2)\in \R \times \C: (\va (y_1, z_2)+iy_1,z_2) \in F\cap U\}.$$ 
It is of measure zero if and only if $F$ is.

When $(y_1,z_2)\in U' \setminus F'$, we may rewrite \eqref{diffpc} as
\begin{equation}
 \label{pcphi}
-\frac14 \fr{\pa^2 \va}{\pa y_1^2} \left \vert \fr{\pa \va}{\pa z_2} \right \vert^2
	- \frac12 \Re \left(i 
	\Big (1+i \fr{\pa \va}{\pa y_1} \Big) \fr{\pa \va}{\pa  z_2}  \fr{\pa^2 \va}{\pa y_1 \pa \bar z_2}  \right)
	- \frac14 \fr{\pa^2 \va}{\pa z_2 \pa \bar z_2} \left  \vert 1- i \fr{\pa \va}{\pa y_1} \right  \vert^2
 \ge 0.
\end{equation}

Define for any pair of continuous functions $\tau=(\tau_1, \tau_2)$, and any distribution $v \in \mathcal D'(U)$ 
such that all its second order partials are measures, 
\begin{equation}
\label{deltatau}
\Delta_\tau v:= |\tau_1|^2 \fr{\pa^2 v}{\pa y_1^2} + 2 \Re \left( i \tau_1 \bar \tau_2 \fr{\pa^2 v}{\pa y_1 \pa \bar z_2} \right)
+ |\tau_2|^2  \fr{\pa^2 v}{\pa z_2 \pa \bar z_2}.
\end{equation}
In particular, let 
\[
\tau(\va):= \left( -\frac12  \fr{\pa \va}{\pa z_2}, \frac12 \left( 1+i \fr{\pa \va}{\pa y_1} \right) \right).
\]

Then the condition \eqref{pcphi} can be written $-\Delta_{\tau(\varphi)} \varphi \ge 0$,
on $U' \setminus F'$.

\subsection{Local result.}

In this context, we can relax the regularity assumption slightly. 
For $k\in \Z_+$, $p\in(0,\infty]$, denote by $W^{k,p}$ the Sobolev space of order $k$ and exponent $p$ (roughly
speaking, the space of distributions with derivatives up to order $k$ in $L^p$). 

\begin{proposition}
\label{localc2}
Let $U=B(0,r)$, and $\Om \subset \C^2$ be a bounded domain and $F$ be a closed subset of
$\pa \Om$ such that $\lambda_3(F)=0$ and $\pa \Om \setminus F$ is $\mathcal C^2-$smooth and locally pseudoconvex.

Suppose that
there exists a neighborhood of $(0,0)$ in $\R\times \C$ and a function 
$\va \in \mathcal C^{1, \al} (U') \cap W^{2,p} (U')$ where $p>3, \al \in (3/p,1),$
such that
\[
\Om \cap U=\{(z_1,z_2)\in U: x_1-\va (y_1, z_2) <0\}, \va(0,0)=\nabla \va (0,0)=0;
\]
suppose that for $(y_1,z_2) \notin F'$ ($F'$ defined from $F$ as above) $\Delta_{\tau(\varphi)} \varphi \le 0$.

Then $\Om \cap U$ is pseudoconvex. 
\end{proposition}

\begin{corollary}
\label{sobobigp}
If we assume $\va \in  W^{2,p} (\mathcal V)$, $p>6$, with the same hypotheses otherwise,  $\Omega$ is pseudoconvex.
\end{corollary}
The corollary follows immediately from the Sobolev Embedding Theorem, which implies
in this instance that $ W^{2,p} \subset \mathcal C^{1, 1-\frac3p}$.

Notice that partial derivatives  $D_j\varphi$ are of the form $-D_j \rho / D_1 \rho$,
and $D_1 \rho$ is bounded from below in a neighborhood of $(0,0)$,
so when $\rho \in \mathcal C^{1,1}$, then $D_j\varphi \in \mathcal C^{0,1}$ (i.e. is Lipschitz),
and in particular, reducing $U'$ if needed, $\va \in \mathcal C^{1, \al} (U') \cap W^{2,p} (U')$, for any $\al<1$, $p>0$.
To prove Theorem \ref{mainsobo}, we cover $\partial \Omega$ with a finite number of balls $U_j=B(a_j,r_j)$, $a_j\in \partial \Omega$,
such that $\Omega \cap U_j = \{z\in U_j: \rho_1(z)<0\}$, with $\rho_1$ of the form given in \eqref{localdef}.
Since by the Proposition, all of those domains are pseudoconvex, so is $\Omega$.

\begin{proof*}{\it Proof of Proposition \ref{localc2}.}

{\it Step 1.}  Extension of the measure. 

By the regularity hypotheses on $\va$, the functions on the left hand side of \eqref{pcphi} define a (real) measure on $U$.
Since it is non-negative and since $\lambda_3(F')=0$ (bilipschitz image of $F$), this measure is actually positive on $U$.
	
\vskip.5cm 
{\it Step 2.} Regularization of the domain. 
Let $\theta \in \mathcal D(\R\times \C)$ be a smooth, non-negative, radial bump function with 
$\int \theta (y_1, z_2) d\lambda_3(y_1,z_2) =1$, supported in the unit ball.
 	 For $\de>0$ we set
$$
	\theta_\de (y_1, z_2):=\fr1{\de^3} \theta (y_1/\de, z_2/\de).
$$
For any open set $V$, let $V^\delta:=\{z\in V: \dist(z,\C^2 \setminus V)>\delta\}$.	Since $\rho$ is $C^1$ on $\mathcal V$ we deduce that:
\begin{enumerate}
\item
$\rho *\theta_\de \in \mathcal C^\infty (\mathcal V^\delta)$  and converges locally uniformly to $\rho$;
\item
All first derivatives of $\rho *\theta_\de$ converge locally uniformly to those of $\rho.$
\end{enumerate}	
Notice that in $U^\delta$,
$$
(\rho *\theta_\de) (z_1,z_2)=x_1- (\va * \theta_\de ) (y_1,z_2).
$$	
As a consequence, $|\nabla (\rho *\theta_\de)|$ is bounded below in a neighborhood of $\partial \Omega$,
for $\de$ small enough.

For $0< \delta \ll \eps$, define 
\begin{equation}
\label{omeps}
\Om_{\ve, \de}:= \{\rho_{\de, \ve} (z):= (\rho*\theta_\de)(z) +\ve |z|^2 +\ve <0 \} .
\end{equation}
For $\eps$ and $\de\ll \eps$ both small enough, $\Om_{\ve, \de}$ is a smoothly bounded domain  and $\Om_{\ve, \de} \subset \{\rho <0\}$.

Now we consider the smaller domain $\Omega^U:= \Omega \cap U$ and define $\Omega^U_{\ve, \de}$ as above.
For any $r'<r$ so that $$B(0,r') \Subset B(0,r)= U,  \delta<r-r',$$
we get $$\Omega^U_{\ve, \de} \cap B(0,r')= \Om_{\ve, \de} \cap B(0,r').$$
 Since balls are obviously 
pseudoconvex, it is enough to see that $\Om_{\ve_j, \de_j}$ is locally pseudoconvex (for 
sequences $\ve_j, \de_j$ going to $0$) to obtain that $\Omega^U$ is.

\vskip.5cm
{\it Step 3.} The domain $\Om_{\ve, \de}$ is pseudoconvex.
\begin{lemma}
\label{convpsh}
Let $v\in \mathcal C^{1, \al} (U') \cap W^{2,p} (U')$ where $p>3, \al \in (3/p,1)$. 
Suppose that $- \Delta_{\tau(\varphi)} v $ is a positive measure.
Then
for any $\eps>0$, there exists $\delta(\eps)$ such that for $\delta<\delta(\eps)$ we have 
$$- \Delta_{\tau(\varphi)} (v * \theta_\de )  \ge - \eps.$$
\end{lemma}

\begin{proof}
The positivity hypothesis implies that $- (\Delta_{\tau(\varphi)} v ) * \theta_\de \ge 0$ as a smooth function.

To make computations more readable, let us denote $$\xi := (\xi_1, \xi_2, \xi_3):= (y_1, \Re z_2, \Im z_2).$$
Then we can write 
\[
\Delta_{\tau(\varphi)} v = \sum_{1\le j \le k \le 3} T_{jk} (\varphi)(\xi) \frac{\partial^2 v}{\partial \xi_j \partial \xi_k}(\xi),
\]
with 
$$\begin{aligned} 
T_{11} &= |\tau_1|^2,\\ 
T_{22} &= T_{33} = \frac14 |\tau_2|^2,\\
T_{12} &= \Im(\bar \tau_1 \tau_2),\\ 
T_{13} &= -\Re(\bar \tau_1 \tau_2),\\
T_{23}&=0,
\end{aligned}$$ 
(here we omit the argument $\varphi$ for the sake of brevity).

Observe that
\begin{multline*}
\Delta_{\tau(\varphi)} (v * \theta_\de )  (y_1,z_2) = \sum_{1\le j \le k \le 3} T_{jk} (\varphi)(\xi)
\frac{\partial^2 }{\partial \xi_j \partial \xi_k}(v * \theta_\de)(\xi)
\\
= \sum_{1\le j \le k \le 3} T_{jk} (\varphi)(\xi) 
\frac{\partial^2 }{\partial \xi_j \partial \xi_k}\int_{U'} v(\xi-\eta) \theta_\delta (\eta)  d\lambda_3(\eta),
\end{multline*}
where $\eta=(\eta_1,\eta_2,\eta_3)$. Differentiating under the integral sign, this equals
\begin{multline*}
= \int_{U'} \Delta_{\tau(\varphi)} v(\xi-\eta) \theta_\delta (\eta)  d\lambda_3(\eta)
\\
+ \int_{U'} \sum_{1\le j \le k \le 3} \left( T_{jk} (\varphi)(\xi) - T_{jk} (\varphi)(\xi-\eta) \right) 
\frac{\partial^2 v}{\partial \xi_j \partial \xi_k}(\xi-\eta) \theta_\delta (\eta)  d\lambda_3(\eta).
\end{multline*}
Since $\varphi \in \mathcal C^{1, \al} (U')$ and $|\eta|\le \delta$ in $\mbox{Supp}(\theta_\delta)$, 
the terms involving the $T_{jk} (\varphi)$ are uniformly bounded 
by $C \delta^{\alpha}$. The second partials of $v$ belong to $L^p(U)$.  Finally, if we let $p'=\frac{p}{p-1}$ be
the conjugate exponent of $p$, $\|\theta_\delta\|_{p'} \lesssim \delta^{-3/p}$, so applying H\"older's inequality,
the second term of our last expression is bounded in modulus by 
\[
C(\varphi) \|v\|_{W^{2,p}}\delta^{\alpha-\frac3p}.
\]
This tends to $0$ as $\delta\to0$; and the first term is nonnegative.
\end{proof}
\n
Applying the lemma to $v=\varphi$, for a finite family of neighborhoods $U$ covering $\partial \Omega$,
we see that the defining function $\rho_{\de, \ve}$ verifies \eqref{diffpc} for $\delta \lesssim \eps^{1/(\alpha-\frac3p)}$.

\vskip.5cm
{\it Step 4.} The domain $\Om^U$ is pseudoconvex.

	Given a compact subset $K$ of $\Om$
	we can find $\ve>\de>0$ such that 
$$K \subset \Om_{\ve, \de} \subset \Omega.$$

	
By Step 3, we can find $\ve_j, \de_j>0$ so small such that 
$\Om_j:=\Om_{\de_{j}, \ve_j}$ is a sequence of pseudoconvex domains with smooth boundary contained in $\Om^U$ such that every compact of $\Om^U$ 
is included in all but a finite number of the $\Om_j.$ Then the continuous function  $-\log \de_{\Om^U}$
 is the pointwise limit of a sequence of the plurisubharmonic functions
  $-\log \de_{\Om_j}$, which are locally uniformly bounded from above, so it must be plurisubharmonic as well. Hence $\Om^U$ is pseudoconvex. 
\end{proof*}

\section{The higher dimensional case}
\label{n3}

We now prove Theorem \ref{mainsobo} in any dimension. We proceed by contradiction. Suppose $\Omega$
is not pseudoconvex. 
Then there exists an affine $2$-complex-dimensional subspace $P$ such that
$\Omega\cap P$ is not pseudoconvex as a subset of $\C^2$ \cite{Hi}, see also \cite{Jb}. Make an affine change of coordinates
such that 
$P=\C^2\times\{0\}$.

We use \cite[Lemma 6]{NT}, itself a variation on \cite[Theorem 4.1.25]{Ho}:

\begin{lemma}
\label{betterHo}
If $D\subset \C^2$ is a bounded non pseudoconvex domain, there exists $p\in \partial D$, $V$ a neighborhood of $p$,
and $\Phi$ a biholomorphism on $V$ such that $\Phi(p)=0$,
$\Phi^{-1}(\overline {\D^2}) \subset V$ and 
 $\Phi^{-1}(G_2)  \subset D$, where 
 $G_2:= \{ z \in \D^2: \Re z_1 < |z_2|^2\}$.
\end{lemma}

Apply the lemma to $D=\Omega\cap (\C^2\times\{0\})$. 
We then consider the biholomorphism given by
$\Phi_1(z):=(\Phi(z_1,z_2),z_3,\dots,z_n)$ on the open set
$V\times \C^{n-2}$, and the open set 
$\Omega_1:= \Omega\cap  (V\times \C^{n-2})$. Notice that $\Phi_1(\Omega_1)\cap (\C^2\times\{0\}) \supset  G_2\times\{0\}$,
and $0\in \partial \Phi_1(\Omega_1)$.
From now on, we denote by $\Omega$ the open set $\Phi_1(\Omega_1)$.

Using a small enough neighborhood $W$ of $0$, 
letting $z':= (z_2,  \dots, z_n)$ and 
$$W':=\{ (y_1, z'): \exists x_1, (x_1+iy_1,z')\in W\},$$
we then have a function $\va \in \mathcal C^{1,1} (W')$
such that
$$\Omega \cap W = \left\{ (z_1,z')\in W: x_1 < \va(y_1,z') \right\},$$
and the conditions above imply that  
\begin{equation}
\label{phige}
\va(0,0)=0, \quad \nabla\va(0,0)=0, \quad
\mbox{ and } \va (0,z_2, 0, \dots, 0) \ge |z_2|^2.
\end{equation}

In order to exploit our zero-measure hypothesis,
we need to take an appropriate two-dimensional slice of $\Omega$.
For $t\in \C^{n -2}$, let $$P(z_1,z_2, t):=p_t(z_1,z_2):= (z_1,z_2, z_2 t).$$
Obviously, $P$ is a 
diffeomorphism on $\C^n\setminus \{z_2=0\}$. Let 
$$\Omega^t:= p_t^{-1}(\Omega) \subset \C^2.
$$
Observe that when $|z_1|, |z_2|,\|t\|$ are small enough, 
and $p_t(z_1,z_2)\in \partial \Omega$, we have 
$$D(p_t)(z_1,z_2) (1,0) \notin T_{p_t(z_1,z_2)} (\partial \Omega).$$
Hence $p_t(\C^2)$ is transverse to $\partial \Omega$. As a consequence, $\partial \Omega^t$ is globally $\mathcal C^{1,1}$-smooth.
Also, use $(y_1,z_2, t)$ as local coordinates on $W\cap \partial \Omega$, we see that $P|_{P^{-1}(\partial \Omega \setminus \{z_2=0\})}$
provides a local diffeomorphism to $\partial \Omega \setminus \{z_2=0\}$. Let 
$$
F^t:= p_t^{-1} (F) \subset \partial \Omega^t, W^t:= p_t^{-1} (W).
$$
 Since $\lambda_{2n-1}(F)=0$, for almost all $t$, we must have $\lambda_3 (F^t \setminus \{z_2=0\})=0$, therefore $\lambda_3 (F^t )=0$.  Choose such a $t$ close enough to $0$: then $\partial \Omega^t$ is $\mathcal C^2$-smooth and pseudoconvex
outside the set $F^t$, which is of measure $0$.

We can then apply the result of Section \ref{n2} to $F^t$ and conclude that $\Omega^t \cap W$ is pseudoconvex. This will lead
to a contradiction. Observe that the defining function for $\Omega^t$ is given by 
$$x_1-\varphi( p_t(y_1,z_2) )=:x_1- \varphi^t(y_1,z_2).$$ 
So using \eqref{phige}, we get 
$$\begin{aligned} 
\varphi^t(0,z_2)&= \varphi (0,z_2, z_2 t) -  \varphi (0,z_2, 0) +  \varphi (0,z_2, 0)\\
&\ge  - C \|t\|^2 |z_2|^2 + |z_2|^2,
\end{aligned}$$
where $C$ depends on the uniform bound for the second derivatives of $\varphi$, which are defined a.e. and bounded.  
Choosing $t$ close enough to $0$, this is bounded below by $c_2|z_2|^2$ with 
$c_2>0$.

Recall that Green's formula applied with a smooth function $u$ and $v(\zeta)= \log \frac{r}{|\zeta|}$ yields
\[
\frac1r \int_{\partial D(0,r)} u (\zeta ) d\lambda_1(\zeta) = u(0) + \int_{D(0,r)} \log \frac{r}{|\zeta|} \Delta u (\zeta) d\lambda_2(\zeta),
\]
see e.g. \cite[(3.1.8)']{Ho}. Approximating by smooth functions, this goes over to $u \in \mathcal C^{1,1}$. 
Since $\va^t(0,0)=0$, by applying the above formula to the function $\xi \to \va^t (0,\xi)$, we obtain
for $r>0$ small enough,
$$
\begin{aligned}
c_2 r^2 &\le \frac1{2\pi} \int_0^{2\pi} \va^t (0,r e^{i\theta}) d\theta\\ 
&= \frac2{\pi} \int_{D(0,r)} \log \frac{r}{|\zeta|} \frac{\partial^2 \va^t}{\partial z_2 \partial \bar z_2}(0, \zeta) d\lambda_2(\zeta).
\end{aligned}
$$
On the other hand, using pseudoconvexity of $\Omega^t,$ we get
$$\Delta_{\tau(\va^t)}\va^t \le 0.$$ 
But,
since $\nabla \varphi (0)=0$ and the first partial derivatives of $\va$ are Lipschitz, we deduce from \eqref{pcphi} that
\begin{equation*}
4 \Delta_{\tau(\va^t)}\va^t (0,z_2)=  \frac{\partial^2 \va^t}{\partial z_2 \partial \bar z_2} (0,z_2) + O ( |z_2|),
\end{equation*}
so that 
\[
c_2 \le \limsup_{r\to0} \frac2{\pi r^2} \int_{D(0,r)} \log \frac{r}{|\zeta|} \frac{\partial^2 \va^t}{\partial z_2 \partial \bar z_2}(0, \zeta) d\lambda_2(\zeta)  \le 0,
\]
and we get a contradiction.

\section{Examples}

In this section, we prove Proposition \ref{c11nonpc} and give a few other examples to show what type of further results could be expected.

\subsection{Proof of Proposition \ref{c11nonpc}.}
\ {}

This argument is inspired from \cite{NPTZ}, in fact it slightly improves Lemma 2.3 in that paper. 
We construct an appropriate function on the interval $[0,1]$.  This technical Lemma will be proved later.

\begin{lemma}
\label{fatstair}
For any $L>0$, there exists a function $F \in \mathcal C^{1,1}(\R)$, supported on $[0,1]$, 
with $\|F\|_\infty \le 1$ such that:
\begin{itemize}
\item
 there exists a point $x_0 \in (0,\frac12)$ and $\delta>0$ such that, for $|t|<\delta$, $F(x_0+t) \ge F(x_0) + tF'(x_0) + Lt^2$;
\item
there exists a dense open set $U\subset [0,1]$ such that $F$ is of class $\mathcal C^2$ and strictly concave on 
each connected component of $U$.
\end{itemize}
\end{lemma}
Accepting the lemma for the moment, we define $\Omega$ as a Hartogs domain over the unit disc, 
\[
\Omega_\varphi:= \{(z,w)\in \C^2: |z|<1, \log |w| < \varphi(z)\}.
\]
Recall that $\Omega$ is pseudoconvex if and only if $-\varphi$ is subharmonic.  
When $\varphi(z)= \frac12 \log (1-|z|^2)$, $\Omega_\varphi$ is just the unit ball. 

Now define $$\Phi(z)=\Phi (x+iy) = F(x-\frac12) \chi (4y),$$ where $F$
is the function obtained in Lemma \ref{fatstair}, $L$ to be chosen, and $\chi$ is a smooth,
even cut-off function on $\R$ such that $0\le \chi \le 1$,
$\mbox{supp}\, \chi \subset (-2,2)$, and $\chi\equiv 1$ on $[-1,1]$.

Let $\Omega:= \Omega_\varphi$, with 
\[
\varphi(z)= \frac12 \log (1-|z|^2) + c_1 \Phi(z),
\]
$c_1$ to be chosen. Note that in a neighborhood of $\{ |z|= 1, w=0\}$ the boundary of $\Omega$ coincides with that
of the unit ball, so is $\mathcal C^\infty$-smooth.  When $|z|<1$, it has the
regularity of $\varphi$; on $[-\frac12;\frac 12]^2$, it has at worst the regularity of $F$,
and outside of that square, it is again $\mathcal C^\infty$-smooth.

Notice that 
$$\Phi(z)=0 \Rightarrow \Delta \varphi (z)=  -2 (1-|z|^2)^{-2}\le -2.$$ So the strict pseudoconvexity of $\Om$ 
near $z$ where $\Phi(z)=0$ is guaranteed.
In general, 
$$\begin{aligned}
\Delta \varphi (z)&=  -2 (1-|z|^2)^{-2} + c_1 F''(x-\frac12) \chi (4y) + 16 c_1 F(x-\frac12) \chi'' (4y)
\\
&\le c_1 F''(x-\frac12) \chi (4y) -1,
\end{aligned}$$   
if we take $c_1 < \frac1{16}\| \chi''\|_\infty$.  The last quantity is negative when $x-\frac12\in U$
and also when $|x| \ge \frac12$ (because then $F(x-\frac12)=0$), so on a dense subset of $\{|z|<1\}$.   

Finally, at the point $z_0:=x_0-\frac12$, $\varphi$ does not admit second derivatives, but taking $\delta<\frac1{10}$, 
we see that for $|\zeta |<\delta$,  
\[
\varphi (z_0+\zeta) \ge    \varphi (z_0) + D\varphi(z_0)\cdot \zeta + \frac12 c_1 L (|\zeta|^2 + \Re (\zeta^2)) - c_3 |\zeta|^2,
\] 
where $c_3>0$ is an absolute constant depending on the  $\log(1-|z|^2)$ term in the function.  Taking
$L$ large enough, we see that $\varphi$ cannot be superharmonic in any neighborhood of the point $z_0$, so $\Omega_\varphi$
is not pseudoconvex.                                                                                              

\begin{proof*}{\it Proof of Lemma \ref{fatstair}.}

We shall construct closed intervals $\{I_{n,i}, J_{n,i}, n \ge 0, 1\le i \le 2^n\}$, such
that $J_{n,i} \subsetneq \mathring{I}_{n,i} $, and
$$E_0:= \bigcap_{n\ge 0} \bigcup_{1\le i \le 2^n} I_{n,i}$$ is a Cantor set (thus totally disconnected), but of positive measure.  
We also construct functions $f_j$ tending to $f$, a non-constant increasing Lipschitz function, constant on each $J_{n,i}$
(a ``devil's staircase'' type function, the growth of which is achieved only on $E_0$).

Let $(\alpha_n)_{n\ge 1} \subset (0,1)$
be a decreasing sequence of positive numbers, such that $$\sum_{n=1}^\infty \alpha_n < \infty.$$

Set 
$$I_{0,1}:= [0,1], J_{0,1}:=[\frac{1-\alpha_1}2, \frac{1+\alpha_1}2], f_0(x):=x$$ 
Then 
$$I_{1,1}:= [0, \frac{1-\alpha_1}2], I_{1,2}:= [\frac{1+\alpha_1}2,1],$$ and $f_1$ is the piecewise affine function given by
 $$f_1(x)=\begin{cases}  \frac12 & x\in J_{0,1},\\
  0  &x=0\\
   1 &x=1\\ 
  \text{affine} &  \text{on each}\ I_{1,i}.
\end{cases}$$
In general, if the intervals $ I_{n,i}=:[a_{n,i},b_{n,i}]$ are known, 
denoting by $|J|$  the length of an interval $J$, and by
$c_{n,i}:= \frac12 (a_{n,i}+b_{n,i})$ the center of $I_{n,i}$, we put 
\[ 
J_{n,i}:= \left[c_{n,i}-\frac12 \alpha_{n+1} |I_{n,i}|, c_{n,i}+\frac12 \alpha_{n+1} |I_{n,i}|\right], 
\]
and denote respectively by $I_{n+1, 2i-1}$ and $I_{n+1, 2i}$ the first and second connected component of 
$I_{n,i}\setminus \mathring{J_{n,i}}$. Then $$|I_{n+1, 2i-1}|= |I_{n+1, 2i}|= \frac12 (1- \alpha_{n+1}) |I_{n,i}|,$$
 which
implies 
$$|I_{n,i}|= 2^{-n}\prod_{i=1}^n (1- \alpha_{n}), \ \forall i \ge 1.$$ 
Given $f_n$, affine increasing on $I_{n,i}$, we obtain $f_{n+1}$ by setting $f_{n+1}(x)= f_n(c_{n,i})$ for all $x\in J_{n,i}$,
and making it continuous and affine on the intervals  $I_{n+1, 2i-1}$ and $I_{n+1, 2i}$.  
An immediate induction shows that
$$f_n(b_{n,i})-f_n(a_{n,i})=2^{-n},$$ 
so that $\|f_n-f_{n+1}\|_\infty \le 2^{-n}$ and the sequence converges 
uniformly to a continuous limit $f$.
On $I_{n,i}$ we have 
$$f_n'(x)= \prod_{i=1}^n (1- \alpha_{n})^{-1}=: \ell_n.$$
Thus $f_n$ is increasing, with 
$$0\le f_n(y)-f_n(x)\le \ell_n(y-x), \  \forall y\ge x.$$ 
Passing to the limit, the convergence condition
means that 
\begin{equation}
\label{flip}
|f(y)-f(x)| \le  |y-x| \prod_{i=1}^\infty (1- \alpha_{n})^{-1} .
\end{equation}
 It also implies
that 
\[
|E_0| = \lim_{n\to\infty} \left| \bigcup_{1\le i \le 2^n} I_{n,i} \right| = \prod_{i=1}^\infty (1- \alpha_{n}) >0.
\]
Note that $$U:= [0,1]\setminus E_0 =  \bigcup_{n\ge 0} \bigcup_{1\le i \le 2^n} \mathring{J_{n,i}}$$ is a dense open subset of $[0,1]$.

We set 
$$
F(x) := \int_0^x (f(t)-t) \chi_{[0,1]}(t) dt.
$$
The integrand is continuous (because it vanishes at $0$ and $1$), and of
Lipschitz class because of \eqref{flip}.
Notice that the construction of $f_n$ implies that 
$$\int_0^1 (f_{n+1}(t)-f_n(t)) dt=0.$$
The function $F$ enjoys the following properties:

\n 
(i) $F\in \mathcal C^{1,1}(\R)$ and $F(x)=0$ whenever $x\le 0$ or $x\ge 1;$

\n 
(ii) For any $(n,i)$,  $F\in \mathcal C^{2}(\mathring{J_{n,i}})$,
and $F''(x)=-1$ there.  

To exhibit a point $x_0$,  on the interval $I_{1,1}$,
consider the function $$g(x):= f(x)- f_1(x)= f(x)- (1- \alpha_{1})^{-1} x.$$ Then $$g(0)=g( \frac{1-\alpha_1}2)=0.$$
This yields
 $$g(x)= f_2(x)-f_1(x) \ \forall x \in J_{1,1}.$$ 
 Hence $g$ takes on some negative values, thus must attain an absolute
minimum at a point $x_0\in \mathring{I_{1,1}}$. Therefore for $x\in I_{1,1}$,
\[
f(x)- (1- \alpha_{1})^{-1} x \ge f(x_0) - (1- \alpha_{1})^{-1} x_0,
\]
so integrating 
\[
F(x_0+s) \ge F(x_0) + s (f(x_0)-x_0) + \frac12 [(1- \alpha_{1})^{-1}-1] s^2,
\]
and we can make $L$ as large as we wish by choosing $\alpha_1$ close to $1$.
\end{proof*}

\subsection{Zygmund regularity and Hausdorff dimension.}
\label{zyg}
\ {}

It has been known since \cite{Ca} that for $0<\alpha<1$,
 a set $E$ in the plane is removable for harmonic functions of class $\Lambda^\alpha= \mathcal C^{0,\alpha}$, if and only if 
$H^\alpha(E)=0$, where $H^\alpha$ stands for Hausdorff measure in dimension $\alpha$. This extends to $\alpha \in [1,2)$
in the following way (see e.g. \cite{Ab}): for an open set $U\subset \R^m$, 
define the Zygmund class as
\begin{multline*}
Z_\alpha (U):= 
\left\{ u\in \mathcal C^0 (U): \exists M<\infty: \right.
\\
\left. \forall x, x+h, x-h \in U,
|u(x+h)+u(x-h)-2u(x)| \le M \|h\|^\alpha \right\}.
\end{multline*}
 Note that for $0<\alpha<1$, $Z_\alpha= \mathcal C^{0,\alpha}$,
for $1<\alpha<2$, $Z_\alpha= \mathcal C^{1,\alpha-1}$.
Then a set $E\Subset U \subset \C$ in the plane is removable for subharmonic functions of class $Z_\alpha(U)$, if and only if 
$H^\alpha(E)=0$.

We will exploit
this to construct a family of examples proving the following.

\begin{proposition}
If $\Omega\subset \C^2$ and $\partial \Omega$ is only assumed $Z_\alpha$-smooth, with $ 0 <\alpha<2$, then the condition $H^{1+\alpha}(F)<\infty$,
for $F$ closed, $F\subset \partial \Omega$,
is not sufficient to imply that $F$ is removable for pseudoconvexity.
\end{proposition}

\begin{proof}
We follow the proof of \cite[Th\'eor\`eme 2]{Ab}, itself recapping
and generalizing the results of \cite{Ca}, \cite{Ul},
\cite{Ve}, among others.

Consider a compact set $E\subset \overline D (0,\frac12)$ such that $0< H^\alpha(E) <\infty$. Such sets can be constructed to be 
totally disconnected. For $\alpha<1$, we may
choose a Cantor-type subset of $[-\frac12,\frac12]$. For any
$\alpha \in (0,2)$, we can construct a ``square'' Cantor set as in
\cite[pp. 315--316]{GM}, with $a_n=a:=4^{-1/\alpha}$.

By Frostman's Lemma, see e.g. \cite[Theorem 8.8]{Ma}, there exists a probability measure $\mu$ supported on $E$ such that for all $r\in(0,1)$,
$\mu(D(0,r)) \le C r^\alpha$.  Then the Riesz potential of $\mu$,
\[
u(z):= -\int_E \log\left| \frac{z-w}{1-z \bar w}\right| d\mu(w),
\]
 vanishes on $\partial D$, is locally bounded and 
satisfies $\Delta u = -\mu$ in the distribution sense.
In particular, for any function $g$ of class $\mathcal C^{1,1}$, $u+g$ cannot be subhmarmonic near any point $z$ where
$$
\limsup_{r\to 0} \fr{\mu(D(z,r))}{r^2} = \infty.
$$
Such points always exist.  Indeed, by \cite[Theorem 8.7(1)]{Ma}, for any $s>\alpha$, $C_s(E)=0$,
where
\begin{multline*}
C_s (A):= \\
\sup \left\{ \left( \int_{A\times A} \frac1{|x-y|^s} d\mu(x) d\mu(y) \right)^{-1} : \mu \mbox{ probability measure on } A\right\},
\end{multline*}
and by the discussion before \cite[Definition 8.3]{Ma}, this is equivalent to the fact that for $s>\alpha$,
there is no probability measure $\nu$ supported on $E$ such that $\nu(D(z,r)) \lesssim r^s$ for all $z\in E$.

Let  $\Omega:= \Omega_\varphi$, with 
\[
\varphi(z)= \frac12 \log (1-|z|^2) - u(z).
\]
Then it is easy to see that $\partial \Omega$ is $Z_\alpha$-smooth, strictly pseudoconvex near points $z\in \partial \Omega 
\setminus F$, where
\[
F:= \{(z,w): z\in E, |w|= \exp(\varphi(z))\};
\]
that $0< H^{1+\alpha}(F)<\infty$, and that $\Omega$ is not pseudoconvex near the points of $F$ where $\mu$
is not dominated by the $2$-dimensional Lebesgue measure,
i.e. where $u$ 
has a singular negative Laplacian.
\end{proof}

\vskip1cm
{\bf Acknowledgements.}  The authors are grateful to Professor Peter Pflug for pointing out an error in a previous version, 
about the description of Grauert and Remmert's results, and showing
us a simplification of the proof of Corollary \ref{alaGR} relying only
on those original results.

We wish to thank the anonymous referee for his/her very careful reading, which weeded out many small errors.

{}


\begin{thebibliography}{}

\bibitem[Ab]{Ab} Abidi, J., {Sur le prolongement des fonctions harmoniques}, Manuscripta Math. {\bf 105} (2001), 471--482.

\bibitem[Ban]{Ban} Banik, A., {\it Visibility domains that are not pseudoconvex},
Bull. Sci. Math. {\bf 193} (2024), 103452. https://doi.org/10.1016/j.bulsci.2024.103452

\bibitem[Ca]{Ca} 
Carleson, L., {\it Selected Problems On Exceptional Sets}, Van Nostrand, Princeton, N.J., 1967.

\bibitem[DD]{DD}
Dinew, S., Dinew, \.Z., {\it On a problem of Chirka}, Proc. Amer. Math. Soc. {\bf 150} (2022), no. 5, 2115--2119.


\bibitem[GM]{GM}
Garnett, J.B., Marshall, D.E., {\it Harmonic Measure}, Cambridge University Press, Cambridge, 2005.

\bibitem[GR]{GR} Grauert, H., Remmert, R., 
{\it Konvexit\"at in der komplexen Analysis. Nicht-holomorph-konvex Holomorphiegebiete und Anwendung auf die Abbildungstheorie}, Comment. Math. Helv. {\bf 31} (1956), 152--183.

\bibitem[Hi]{Hi} Hitotumatu, S., {\it On some conjectures concerning pseudo-convex domains}, J. Math. Soc. Jpn. {\bf 6} (1954), 177--195.

\bibitem[Ho]{Ho} H\"ormander, L., {\it Notions of convexity}, Birkh\"auser, Boston-Basel-Berlin, 1994.

\bibitem[Jb]{Jb} Jacobson, R., {\it Pseudoconvexity is a two-dimensional phenomenon}, arXiv:0907.1304

\bibitem[Ma]{Ma} Mattila, P., {\it Geometry of sets and measures in Euclidean space}, Cambridge University Press, Cambridge, 1995.

\bibitem[NOT]{NOT} Nikolov, N., \"Okten, A.Y., Thomas, P.J., 
{\it Visible $\mathcal C^2$-smooth domains are pseudoconvex}, Bull. Sci. Math.
(2024), 103525. https://doi.org/10.1016/j.bulsci.2024.103525

\bibitem[NPTZ]{NPTZ} Nikolov, N., Pflug, P., Thomas, P.J., Zwonek, W., {\it On a local characterization of pseudoconvex domains},
Indiana Univ. Math. J. {\bf 58}, no. 6 (2009), 2661--2672.

\bibitem[NT]{NT} Nikolov, N., Thomas, P.J., {\it Quasi Triangle Inequality for the Lempert function},
arXiv:2503.19754v2


\bibitem[Ul]{Ul} Ullrich, D. C., {\it Removable Sets for Harmonic Functions}, Michigan Math. J. {\bf 38} (1991), 467--473.

\bibitem[Ve]{Ve} Verdera, J., {\it $C^m$ approximation by solutions of elliptic equations, and Calder\'on–
Zygmund operators}, Duke Math. J. {\bf 55} (1987), 157--187.

\bibitem[ZZ]{ZZ} Zaitsev, D., Zampieri, G., {\it Domains of holomorphy with edges and lower dimensional boundary singularities},
Complex Var. Theory Appl. {\bf 47} (2002), no. 11, 969--979.

\end{thebibliography}
\end{document}